\documentclass[12pt]{article}

\usepackage{amsmath}
\usepackage{amsthm}
\usepackage{amssymb}
\usepackage{color,epsfig}
\usepackage{fullpage}
\usepackage{enumerate}
\usepackage{paralist}
\usepackage{hyperref}
\usepackage{relsize}
\usepackage{exscale}

\usepackage{tikz}
\usetikzlibrary{patterns}
\usepackage{subfig}

\usepackage{changes}
\usepackage{changes}
\definechangesauthor[color=red]{mk}
\definechangesauthor[color=blue]{dm}
\usepackage{ifpdf}

\newtheorem{theorem}{Theorem}
\newtheorem{proposition}[theorem]{Proposition}

\newtheorem{corollary}[theorem]{Corollary}
\newtheorem{lemma}[theorem]{Lemma}
\newtheorem{remark}[theorem]{Remark}
 
\newcommand{\dist}{\mathrm{d_h}}
\newcommand{\distp}{\mathrm{d'_h}}
\newcommand{\distpp}{\mathrm{d''_h}}

\newcommand{\calB}{\mathcal{B}}
\newcommand{\calC}{\mathcal{C}}

\newcommand{\calG}{\mathcal{G}}
\newcommand{\calH}{\mathcal{H}}

\newcommand{\calR}{\mathcal{R}}
\newcommand{\calS}{\mathcal{S}}

\newcommand{\calW}{\mathcal{W}}

\newcommand{\EE}{\mathbb{E}}
\newcommand{\HH}{\mathbb{H}}
\newcommand{\NN}{\mathbb{N}}
\newcommand{\PP}{\mathbb{P}}
\newcommand{\RR}{\mathbb{R}}

\newcommand{\ZZ}{\mathbb{Z}}

\renewcommand{\Pr}{\mathbf{P}}
\newcommand{\Ex}{\mathbf{E}}

\newcommand{\poimod}{\mathrm{Poi}}

\begin{document}

\title{On the second largest component of random hyperbolic graphs}

\author{Marcos Kiwi\thanks{Depto.~Ing.~Matem\'{a}tica \&
  Ctr.~Modelamiento Matem\'atico (CNRS UMI 2807), U.~Chile. 
  Beauchef 851, Santiago, Chile, Email: \texttt{mkiwi@dim.uchile.cl}. Gratefully acknowledges the support of 
  Millennium Nucleus Information and Coordination in Networks ICM/FIC P10-024F
  and CONICYT via Basal in Applied Mathematics.} \\
\and 
Dieter Mitsche\thanks{Institut Camille Jordan, Universit\'e Jean Monnet, Email: \texttt{dieter.mitsche@univ-st-etienne.fr}. The second author has been supported by IDEXLYON of Universit\'{e} de Lyon (Programme Investissements d'Avenir ANR16-IDEX-0005).}}

\maketitle 

\begin{abstract}
  We show that in the random hyperbolic graph model as formalized by~\cite{GPP12} in the most interesting range of $\frac12 < \alpha < 1$ the size of the second largest component is $\Theta((\log n)^{1/(1-\alpha)})$. Our research is motivated by the
question raised in [BFM13] regarding the uniqueness of linear size
components in random hyperbolic graphs which naturally leads to the
question regarding the size of the second largest component.
We also show that for $\alpha=\frac12$ with constant probability the corresponding size is $\Theta(\log n)$, whereas for $\alpha=1$ it is $\Omega(n^{b})$ for some $b > 0$.
\end{abstract}

\section{Introduction}\label{sec:intro}
The model of random hyperbolic graphs introduced by Krioukov et al.~\cite{KPKVB10} has attracted quite a bit of interest due to its key properties also observed in large real-world networks. One convincing demonstration of this fact was given by Bogu\~n\'{a} et al.~in~\cite{BPK10}
  where a compelling (heuristic)
  maximum likelihood fit of autonomous systems of the internet graph in hyperbolic 
  space was computed. A second reason for why the model  
  initially caught attention is due to the 
  experimental results reported by Krioukov et al.~\cite[\S~X]{KPKVB10} confirming that the model exhibits
  the algorithmic small-world phenomenon established
  by the groundbreaking letter forwarding experiment of Milgram from 
  the 60's~\cite{Mil67}.

\medskip
Another important aspect of the random graph model introduced 
  in~\cite{KPKVB10} is its mathematically elegant specification and the 
  fact that it is amenable to mathematical analysis.
This partly explains why the model has been studied not only empirically 
  by the networking community but also analytically by theoreticians.
For the latter, it is natural to first consider those issues that
  played a crucial role in the development of the theory of 
  other random graph models.
Among these, the Erd\H{o}s-R\'{e}nyi random graph model is undisputedly 
  the most relevant. 
One of the most, if not the most, studied aspect of the Erd\H{o}s-R\'{e}nyi model
  is the evolution (as a function of the graph density) 
  of the size and number of its connected components~\cite{ER60}, specially the 
  size of the largest one, but also the size of the second largest.
These studies have played a crucial role in the development of 
  mathematical techniques and significantly contributed to the understanding
  of the Erd\H{o}s-R\'{e}nyi random graph model.
For the random hyperbolic graph model, the study of the largest 
  component's size was started by Bode, Fountoulakis and M\"{u}ller~\cite{BFM13}
  and recently refined by Fountoulakis and M\"{u}ller~\cite{FMLaw}.
A logarithmic lower bound and polylogarithmic upper bound 
  for the size of the second largest component of random
  hyperbolic graphs (when $\frac12<\alpha<1$) were first established 
  in~\cite{KM18}.
In this paper we improve on these bounds and determine the 
  precise order of the size 
  of the second largest component of random hyperbolic graphs.

\bigskip\noindent
\textbf{Model specification:}
In the original model of Krioukov et al.~\cite{KPKVB10} an $n$-vertex size graph $G$ was obtained by first randomly choosing $n$ points in $B_{O}(R)$ (the disk of radius $R=R(n)$ centered at the origin $O$ of the hyperbolic plane).
From a probabilistic point of view it is arguably more natural to consider the Poissonized version of this model. Formally, the Poissonized model 
is the following (see also~\cite{GPP12} for the same description in the uniform model): for each $n \in \NN$, consider a Poisson point process on the hyperbolic disk of radius $R :=2 \log (n/\nu)$ for some positive constant $\nu \in \RR^+$ ($\log$ denotes here and throughout the paper the natural logarithm) and denote its point set by $V$ (the choice of $V$ is due to the fact that we will identify points of the Poisson process with vertices of the graph). The intensity function at polar coordinates $(r,\theta)$ for 
  $0\leq r< R$ and $0 \leq \theta < 2\pi$ is equal to
\[
g(r,\theta) := \nu e^{\frac{R}{2}}f(r,\theta),
\]
where $f(r,\theta)$ is the joint density function with $\theta$ chosen uniformly at random in the interval $[0,2\pi)$ and independently of $r$, which is chosen according to the density function
\begin{align*}
f(r) & := \begin{cases}\displaystyle
   \frac{\alpha\sinh(\alpha r)}{\cosh(\alpha R)-1}, &\text{if $0\leq r< R$}, \\
   0, & \text{otherwise}.
  \end{cases}
\end{align*}
Note that this choice of $f(r)$ corresponds to the uniform distribution inside a disk of radius $R$ around the origin in a hyperbolic plane of curvature $-\alpha^2$. Identify then the points of the Poisson process with vertices
(that is, identify a point with polar coordinates $(r_v,\theta_v)$ with vertex $v\in V$) and make the following graph $G=(V,E)$: for $u, u'\in V$, $u \neq u'$, there is an edge with endpoints 
  $u$ and $u'$ provided the distance (in the hyperbolic plane) between
  $u$ and $u'$ is at most $R$, i.e.,  
  the hyperbolic distance
  between $u$ and $u'$, denoted by 
  $\dist:=\dist(u,u')$,
  is such that  $\dist\leq R$ where $\dist$ is obtained by solving 
\begin{equation}\label{eqn:coshLaw}
\cosh \dist := \cosh r_u\cosh r_{u'}-
  \sinh r_u\sinh r_{u'}\cos( \theta_u{-}\theta_{u'}).
\end{equation}

For a given $n \in \NN$, we denote this model by 
  $\poimod_{\alpha,\nu}(n)$.
Note in particular that 
\[
\iint g(r,\theta) d\theta dr 
  = \nu e^{\frac{R}{2}}=n,
\]
and thus  $\EE|{V}|=n.$ 
  The main advantage of defining $V$ as a Poisson point process is
  motivated by the following two properties: the number of points of
  $V$ that lie in any region $A \subseteq B_O(R)$ follows a Poisson
  distribution with mean given by $\int_A g(r,\theta) drd\theta=n
  \mu(A)$, and the numbers of points of $V$ in disjoint
  regions of the hyperbolic plane are independently distributed.

The restriction $\alpha>\frac12$ and the role of $R$, informally speaking,
  guarantee that the resulting graph has bounded average degree (depending
  on $\alpha$ and $\nu$ only): if $\alpha<\frac12$, then the degree sequence is so 
  heavy tailed that this is impossible (the graph is with high probability connected in this case, as shown in~\cite{BFM13b}), and if $\alpha>1$, then
  as the number of vertices grows,
  the largest component of a random hyperbolic graph has sublinear 
  order~\cite[Theorem~1.4]{BFM15}.
In fact, although some of our results hold for a wider range of $\alpha$, 
  we will always assume $\frac12<\alpha<1$; only in the concluding remarks we discuss the cases $\alpha=\frac12$ and $\alpha=1$.

It is known that for $\frac12 < \alpha < 1$, with high probability the 
  graph $G$ has a linear size 
  component~\cite[Theorem~1.4]{BFM15}  
  and all other components are of
  polylogarithmic order~\cite[Corollary 13]{km15}, 
  which justifies referring to the
  linear size component as \emph{the giant component}.
Implicit in the proof of~\cite[Theorem 1.4]{BFM15} is that the giant component of a
  random hyperbolic graph $G$ is the one
  that contains all vertices whose radial coordinates are at most $\frac{R}{2}$. More
  precise results including a law of large numbers for the largest
  component in these networks were established recently
  in~\cite{FMLaw}. 

\bigskip\noindent
\textbf{Main result and proof overview:}
in this paper we determine the exact order of 
  the size of the second largest component, which we denote by $L_2(G)$. 

\bigskip
We say that an event holds \emph{asymptotically almost surely (a.a.s.)}, if it holds with probability tending to $1$ as $n \to \infty.$ We use the standard Bachmann-Landau notation for the asymptotic behaviour of sequences: For two sequences $(a_n)_{n\ge 0}$ and $(b_n)_{n\ge 0}$,  we write $a_n = O(b_n)$ to denote the existence of a constant $C > 0$ and a non-negative integer $n_0$ such that $|a_n| \le C|b_n|$ for all $n\geq n_0$. Moreover, we write $a_n = \Omega(b_n)$ if
$b_n = O(a_n)$, and $a_n = \Theta(b_n)$ if both $a_n = O(b_n)$ and $a_n =\Omega(b_n)$. Finally, $a_n=o(b_n)$ if for every constant $C>0$ there is a non-negative integer $n_0$ such that $|a_n| \le C |b_n|$ for all $n\geq n_0$, and moreover $a_n=\omega(b_n)$ if $b_n=o(a_n)$.
The main result of this paper is the following:
\begin{theorem}\label{thm:main}
  Let $\frac12 < \alpha < 1$.
  If $G=(V,E)$ is chosen according to $\poimod_{\alpha,\nu}(n)$, then a.a.s.,
\[
L_2(G)=\Theta(\log^{\frac{1}{1-\alpha}} n).
\]
Moreover, for some sufficiently small constant $b>0$, there are 
  $\Omega(n^{b})$ components in $G$, each one of size 
  $\Theta(\log^{\frac{1}{1-\alpha}} n)$.
\end{theorem}
To establish the lower bound, we partition the disk into sectors, so that close to the central axis of each sector, one can find a chain (a path) of vertices within a certain distance from the boundary so that the expected number of vertices with larger radius and in the same sector is of the desired order. While it is relatively easy to show that a constant fraction of these large radii vertices indeed connects to the chain (hence, belong to the same connected component), it is more work to show that none of these vertices in fact is connected to the giant component. Technically, this is tedious since vertices at all radii might potentially be connected to the giant component; vertices with smaller radii might be more dangerous to have neighbors with smaller radii, whilst vertices with larger radii (close to the boundary of $B_{O}(R)$)  might be more dangerous to being reachable from vertices of larger radii that connect to the giant component.

\medskip

An original aspect of our lower bound 
  analysis consists in identifying ``walls'', that is, 
  regions $\calW$, inside $B_{O}(R)$ and close to 
  its boundary (specifically, a 
  collection of connected points at 
  distance at least $\ell:=R-O(\log R)$ 
  from the origin) which satisfy the following conflicting
  properties: 
  (i) they do not contain vertices, and
  (ii) for a sector  
  $\Phi$ of $B_{O}(R)$ strictly containing $\calW$,
  the region $\Phi\setminus B_{O}(\ell)$ 
  is partitioned into connected regions
  $\calW',\calW,\calW''$
  in such a way that the hyperbolic distance between a point in 
  $\calW'$ and a point in $\calW''$ is greater than $R$. 
The abundance of walls coupled
  with the fact that the subgraph of $G$ induced
  by the vertices in $B_{O}(R)\setminus B_{O}(\ell)$ contains
  many vertices (belonging to connected components which we refer to as \emph{pre-components})
  reduces the problem of bounding $L_2(G)$ from below to one of showing that there are sectors of $B_O(R)$, say $\Phi$,
  for which $\Phi\cap B_{O}(R)\setminus B_{O}(\ell)$ contains 
  a relatively large connected component while $\Phi\cap B_{O}(\ell)$ is
  unlikely to contain vertices of $G$ 
  (these latter regions are the
  ones where neighbors of pre-components can potentially lie).

Interestingly, the mentioned abundance of walls also partly explains 
the hierarchical structure close to the boundary of 
  $B_{O}(R)$ that random hyperbolic graphs exhibit 
  (see Figure~\ref{fig:forest}). 
\begin{figure}[ht]
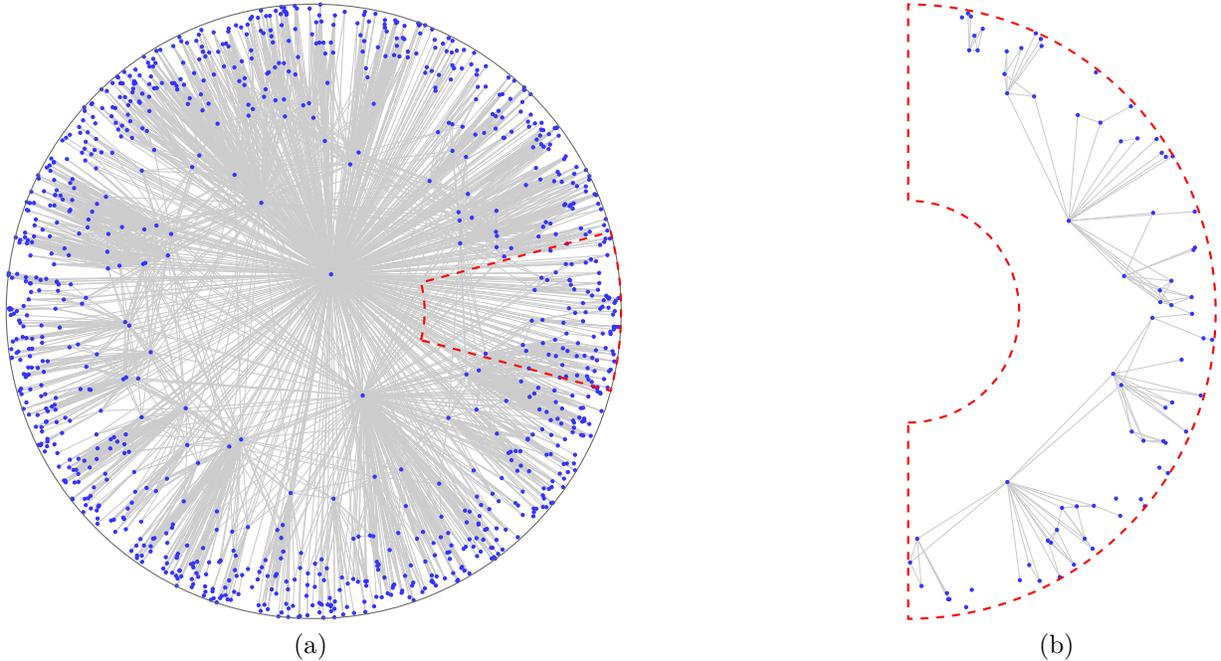

\def\vertRad{1pt}
\subfloat[][]{\centering\input{fig1.inc}}
\hfill
\subfloat[][]{\centering\input{fig2.inc}}
\caption{(Left) An instance $G$ of Krioukov et.~al.'s 
  random hyperbolic graph model with parameters $n=1000$, $\alpha=0.7$,
  and $\nu=1.1$.
  (Right) The subgraph of $G$ induced by the vertices inside the dashed region shown on the left side, where angular coordinates have been scaled by a factor of $6$ in order to better elicit the hierarchical structure of the induced graph.}\label{fig:forest}
\end{figure}

\medskip 
  The upper bound of Theorem~\ref{thm:main}, easier than the lower bound, makes use of the fact that all vertices that are not too close to the boundary of $B_O(R)$ belong to the giant component. We can thus find in every sector of not too big angle a vertex belonging to the giant component, and by simple known geometric properties of random hyperbolic graphs any other component must be squeezed between two such sectors. Since the number of vertices in such a sector is concentrated, we get an upper bound on the size of the second component.

\medskip
To conclude our study of the size of the second largest component of 
  random hyperbolic graphs we consider the relevant remaining cases where 
  $\alpha=\frac12$ or $\alpha=1$.
In the former case, we show that a.a.s.~every vertex of the 
  second largest component must be within
  $C=\Theta(1)$ of the boundary of $B_{O}(R)$.
Moreover, by some geometric considerations, such a component must 
  be contained in a sector $\Phi$ of $B_{O}(R)$ 
  for which $\Phi\cap B_{O}(R-C)$ does
  not contain vertices of $G$.
An analysis of the likely maximum angle such a sector $\Phi$ can have 
  and of the number of vertices that can be found in 
  $\Phi \setminus B_{O}(R-C)$ yields the following:
\begin{proposition}\label{p:alphaHalf}
For $\alpha=\frac12$ and $\nu$ small enough, with constant probability, $L_2(G)=\Theta(\log n)$.
\end{proposition}
Observe that this result is rather surprising, as the size of the second largest component is discontinuous as $\alpha$ tends to $\frac12$ from above: the formula $\Theta((\log n)^{\frac{1}{1-\alpha}})$ would suggest a size of $\Theta(\log^2 n)$, contradicting Proposition~\ref{p:alphaHalf}. For the $\alpha=1$ case, we show that there is a $\frac12<\lambda<1$ for which a.a.s.~there is a vertex of degree $\Theta(n^{1-\lambda})$ that belongs to a component separated from the giant (if the latter exists), so we obtain the following:
\begin{proposition}\label{p:alphaOne}
For $\alpha=1$ there exists $\gamma$, $0<\gamma<1$ such that 
  a.a.s.,~$L_2(G)=\Omega(n^{\gamma})$. 
Moreover, there exists some $0 < \delta < \gamma$ so that for some sufficiently small constant $b>0$, a.a.s.~there are $\Omega(n^b)$ components in $G$, each one of size $\Omega(n^{\delta})$.
\end{proposition}
\bigskip\noindent
\textbf{Related work:}
Although the random hyperbolic graph model was relatively
  recently introduced~\cite{KPKVB10}, several of its key properties have already been established. As already mentioned, in~\cite{GPP12}, the degree distribution, the expected value of the maximum degree and global 
  clustering coefficient were determined, and in~\cite{BFM15}, the existence of a giant component as a function of $\alpha$.
  
  The threshold in terms of $\alpha$ for the connectivity of random
  hyperbolic graphs was given in~\cite{BFM13b}. 
Concerning diameter and graph distances,
  except for the aforementioned papers of~\cite{km15} and~\cite{fk15},
  the average distance of two points belonging to the giant component
  was investigated in~\cite{ABF}. 
Results on the global clustering coefficient of the so called
  binomial model of random hyperbolic graphs were obtained
  in~\cite{CF16}, and on the evolution of graphs on more general
  spaces with negative curvature in~\cite{F12}. 
Finally, the spectral gap of the Laplacian of this model was studied 
  in~\cite{KM18}.

The model of random hyperbolic graphs for $\frac12 < \alpha < 1$ is very similar to two different models
studied in the literature: the model of inhomogeneous long-range
percolation in $\ZZ^d$ as defined in~\cite{Remco}, and the
model of geometric inhomogeneous random graphs, as introduced
in~\cite{BKL19}. In both cases, each vertex is given a weight, and
conditionally on the weights, the edges are independent (the presence
of edges depending on one or more parameters). In~\cite{Remco} the
degree distribution, the existence of an infinite component and the
graph distance between remote pairs of vertices in the model of
inhomogeneous long-range percolation are analyzed. On the other hand,
results on typical distances, diameter, clustering coefficient,
separators, and existence of a 
  giant component in the model of geometric
inhomogeneous graphs were given in~\cite{BKL1,BKL19}, bootstrap
percolation in the same model was studied in~\cite{KL} and 
  greedy routing in~\cite{BKLMM17}. Both models
are very similar to each other, and similar results were obtained in
both cases. The latter model generalizes random hyperbolic graphs.

\bigskip\noindent\textbf{Notation:}
All asymptotic notation in this paper is with respect to $n$. Expressions given in terms of other variables such as $O({R})$ are still asymptotics with respect to $n$, since these variables still depend on $n$.
We say that an event holds \emph{with extremely high probability
    (w.e.h.p.)}, if for every $c > 0$, there exists an $n_0:=n_{0}(c)$
  such that for every $n\geq n_0$ the event holds with
  probability at least $1-O(n^{-c})$. 
Note that the union of polynomially (in $n$) many events (where the degree 
  of the polynomial is not allowed to depend on $c$) that hold w.e.h.p.~is 
  also an event that holds w.e.h.p. In what follows, any union bound 
  is over at most $O(n^2)$ many events. 


\section{Preliminaries}\label{sec:prelim}
In this section we collect some of the known properties concerning 
  random hyperbolic graphs. 
  
\medskip
By the hyperbolic law of cosines~\eqref{eqn:coshLaw}, 
  the hyperbolic triangle formed by the geodesics 
  between points $p'$, $p''$, and  $p$, with opposing side segments of length 
  $\distp$, $\distpp$, and $\dist$ respectively,
  is such that the angle formed at $p$ is:
\begin{equation*}\label{eqn:angle}
\theta_{\dist}(\distp,\distpp) = 
\arccos\Big(\frac{\cosh \distp\cosh \distpp-\cosh \dist}{\sinh \distp\sinh \distpp}\Big).  
\end{equation*}
Clearly, $\theta_{\dist}(\distp,\distpp) = \theta_{\dist}(\distpp,\distp)$. 
\begin{remark}\label{rem:monotonTheta}
Recall that $\cosh(\cdot)$ is at least $1$ and strictly increasing
  in $\RR^+$.
Moreover, $\cosh^2 x-\sinh^2 x=1$.
Hence, if $0<x, y\leq R$, then
\[
\frac{\partial}{\partial x}\Big(\frac{\cosh x\cosh y - \cosh R}{\sinh x\sinh y}\Big)
  = \frac{-\cosh y+ \cosh R\cosh x}{\sinh^2 x\sinh y}
  > \frac{\cosh R-\cosh y}{\sinh^2 x\sinh y}
  \geq 0.
\]
Since $\arccos(\cdot)$ is strictly decreasing, it follows that 
  $\theta_{R}(\cdot,y)$ is strictly decreasing for fixed $0<y\leq R$.
By symmetry, a similar claim holds for 
  $\theta_{R}(x,\cdot)$.
\end{remark}

Next, we state a very handy approximation for $\theta_{R}(\cdot,\cdot)$.
\begin{lemma}[{\cite[Lemma 3.1]{GPP12}}]\label{lem:aproxAngle}
If $0\leq\min\{\distp,\distpp\}\leq R \leq \distp+\distpp$, then
\[
\theta_{R}(\distp,\distpp) = 2e^{\frac{1}{2}(R-\distp-\distpp)}\big(1+\Theta(e^{R-\distp-\distpp})\big).
\]
\end{lemma}
\begin{remark}\label{rem:aproxAngle}
We will use the previous lemma also in this form: let
  $p'$ and $p''$ be two points at distance $R$ from each other
  such that $r_{p'},r_{p''} > \frac{R}{2}$ and $\min\{r_{p'},r_{p''}\}  \leq R$.
Then, taking $\distp=r_{p'}$
  and $\distpp=r_{p''}$ in Lemma~\ref{lem:aproxAngle}, we get
\[
\theta_{R}(r_{p'},r_{p''}) 
  := 2e^{\frac{1}{2}(R - r_{p'} - r_{p''})}\big(1+\Theta(e^{R-r_{p'}-r_{p''}})\big).
\]
\end{remark}

\medskip
Throughout, we will need estimates for measures of 
  regions of the hyperbolic plane, and more specifically, for regions obtained by performing some set algebra
  involving a few balls.
For a point $p$ of the hyperbolic plane $\HH^2$, 
  the ball of radius $\rho$ centered at $p$ will be denoted by
  $B_{p}(\rho)$, i.e., 
  $B_{p}(\rho) := \{q\in\HH^2 : \dist(p,q)\leq\rho\}$.

Also, we denote by $\mu(S)$ the measure of a set 
  $S \subseteq \HH^2$, i.e.,
  $\displaystyle\mu(S) := \int_{S}f(r,\theta)drd\theta$.

Next, we collect a few results for such measures.
\begin{lemma}[{\cite[Lemma~3.2]{GPP12}}]\label{lem:muBall} 
If $0\leq \rho< R$, then
  $\mu(B_{O}(\rho)) = e^{-\alpha(R-\rho)}(1+o(1))$.
\end{lemma}
A direct consequence of Lemma~\ref{lem:muBall} is
\begin{corollary}\label{cor:medidalb} 
If $0 \leq \rho'_O < \rho_O < R$, then
\[
\mu(B_{O}(\rho_{O})\setminus B_{O}(\rho'_{O}))
  = e^{-\alpha(R-\rho_{O})}(1-e^{-\alpha(\rho_{O}-\rho'_{O})}+o(1)).
\]
\end{corollary}
By standard estimates for Poisson random variables,
we have the following lemma:%
\begin{lemma}[{\cite[Lemma~12]{KM18}}]\label{lem:Poisson}
Let $V$ be the vertex set of a graph chosen according to $\poimod_{\alpha,\nu}(n)$. 
For every $c>0$, there is a sufficiently large
  constant $c'=c'(c)$ 
  such that if $S\subseteq B_{O}(R)$ with $\mu(S)\ge c'\log n/n$,
  then with probability at least $1-n^{-c}$,
  $|S\cap V|=\Theta (n \mu(S)).$ If moreover $S \subseteq B_{O}(R)$ is such that $\mu(S) = \omega(\log n/n)$, then w.e.h.p.~$|S\cap V|=\Theta (n \mu(S)).$ 
\end{lemma}
We need one more lemma.

\begin{lemma}[{\cite[Lemma~9]{fk15}}]\label{lem:FK15}
Let $p,p',p''\in B_{O}(R)$ be such that $\theta_{p} \le \theta_{p'} \le \theta_{p''}$ and let $\dist(p,p'')\leq R$. 
Then the following holds: 
\begin{enumerate}[(i).-]
\item 
if $r_{p'} \le \min\{r_{p},r_{p''}\}$,
  then $\dist(p,p'),\dist(p',p'')\leq R$.
\item if $r_{p'} \le r_{p''}$, then $\dist(p,p') \le R$.
\end{enumerate}
\end{lemma}

\section{Intermediate regime of $\alpha$}\label{sec:proofs}
In this section we prove the main result of this article which concerns
  the regime where $\alpha$ takes values strictly between $\frac12$ and $1$. 
Since our results are asymptotic, we may and will ignore floors in the 
  following calculations, and assume that certain expressions such as 
  $R-\frac{\log R}{1-\alpha}$, $R-\frac{\log R}{1-\alpha}-L$ for some constant $L$ or the like are integers, if needed.
  When working with a Poisson point process $V$, for a positive 
  integer $\ell$, we refer to the vertices of $G$ that belong to 
  $B_{O}(\ell)\setminus B_{O}(\ell-1)$ as the $\ell$-th \emph{band} or \emph{layer} and 
  denote it by $V_{\ell}:=V_{\ell}(G)$, i.e.,  
  $V_{\ell}:=V \cap B_{O}(\ell)\setminus B_{O}(\ell-1)$.
Throughout this section we always assume that $\frac12<\alpha<1$.

\subsection{Upper bound}
We start with some observations that simplify arguing about the giant component 
  of random hyperbolic graphs. 
Henceforth, 
  we call \emph{center component} the connected component 
  containing all vertices of a random hyperbolic graph that are 
  within distance $\frac{R}{2}$ of the origin (since all the latter 
  vertices are within distance $R$ of each other, they belong to
  the same connected component).
Concerning the relation between the giant component and the center component,
  the following is known:
\begin{proposition}[Bode et al.~\cite{BFM15}]\label{prop:bfm}
W.e.h.p.~the giant and center component coincide.
\end{proposition}
Next we establish that it is likely that vertices bounded
  away from the boundary of $B_{O}(R)$ belong to the center component. A similar but slightly weaker result was already proven in~\cite{KM18}.
\begin{lemma}\label{lem:HtoG}
Let $\ell(L):= R-\frac{\log R}{1-\alpha}- L$ and $G=(V,E)$ 
  be chosen according to $\poimod_{\alpha,\nu}(n)$. 
For every $c>0$, there is a 
  sufficiently large constant $L:=L(c)>0$ such that 
  with probability at least $1-O(n^{-c})$, all vertices in 
  $V \cap B_O(\ell)$, $\ell=\ell(L)$, belong to the center component.
\end{lemma}
\begin{proof}
It suffices to show that for a sufficiently large $L$ and 
  every vertex $v\in V_{i}$ with $\frac{R}{2} \le i\le \ell$ with 
  probability at least $1-O(n^{-c})$, there exists a path connecting 
  $v$ to a vertex in $V \cap B_O(\frac{R}{2})$. 
Taking a union bound, and iterating the argument
  with $i-1$ instead of $i$ until $i=\frac{R}{2}$, 
  it is enough to show (as proved next) that 
  for a fixed vertex $v\in V_i$ with $i$ as before,
  with probability at least $1-O(n^{-(c+1)})$, vertex $v$ has a neighbor 
  in $V_{i-1}$. 

By Remark~\ref{rem:aproxAngle}, $v$ is connected to vertex $u \in V_{i-1}$ if the angle at the origin between $u$ and $v$ is 
  $O(\theta_R(i,i))$. 
By Corollary~\ref{cor:medidalb}, we have
\[
\mu(B_{v}(R)\cap B_{O}(i-1)\setminus B_O(i-2)) =\Theta(e^{-\alpha(R-i)}e^{\frac12(R-2i)})=\Theta(e^{(1-\alpha)(R-i)}/n).
\]
Since $\alpha < 1$, this expression is clearly decreasing in $i$, and plugging in our upper bound on~$i$, we obtain
\[
\mu(B_{v}(R)\cap B_{O}(i-1)\setminus B_O(i-2)) =\Omega(e^{(1-\alpha)(R-\ell)}/n)=\Omega(\log n/n),
\]
where the constant hidden in the asymptotic expression can be made arbitrarily large by choosing $L$  large enough so that applying Lemma~\ref{lem:Poisson} guarantees that with probability at least $1-O(n^{-(c+1)})$, vertex $v$ has 
  $\Omega(\log n)$ neighbors in $V_{i-1}$. 
By definition, $v$ is connected by an edge to any such vertex, and hence in 
  particular with probability at least $1-O(n^{-(c+1)})$, vertex $v$ 
  has a neighbor in $V_{i-1}$. 
\end{proof}

Define next a $\phi$-sector $\Phi$ to be a sector of $B_O(R)$, that contains all points in $B_O(R)$ making an angle of at most $\phi$ at the origin with an arbitrary but fixed reference point. 

We deduce from the previous lemma that in any not too small angle there will be at least one vertex belonging to the giant component:
\begin{lemma}\label{lem:maxAngle}
  For every $c>0$, if $L:=L(c)$ and $L':=L'(c)$ are sufficiently large
  constants, then, for 
  $\ell=\ell(L):= R-\frac{\log R}{1-\alpha}- L$ 
  and $\phi=\phi(L') := \frac{L'}{n}(\log n)^{1/(1-\alpha)}$,
  with probability at least $1-O(n^{-c})$,
  every $2\phi$-sector $\Phi$ contains at least one vertex 
  $v \in V_{\ell}$.
\end{lemma}
\begin{proof}
Partition $B_O(R)$ into $\phi$-sectors $\Phi_1, \ldots, \Phi_{2\pi/\phi}$. 
By Corollary~\ref{cor:medidalb}, we get
\[
\mu(\Phi_i\cap B_O(\ell)\setminus B_{O}(\ell-1))
  = \Theta(\phi e^{-\alpha(R-\ell)})=\Theta(\log n/n).
\]
For $L'$ sufficiently large, the constant hidden in the asymptotic notation 
  can be made as large as required by Lemma~\ref{lem:Poisson} to get that, 
  with probability at least $1-O(n^{-(c+1)})$, the number of vertices in 
  $V_{\ell}\cap\Phi_{i}$ is $\Theta(\log n)$. 
By taking a union bound over all  $\phi$-sectors $\Phi_i$ 
  (there are $2\pi/\phi=O(n)$ of them), this holds with probability at 
  least $1-O(n^{-c})$ in all of them simultaneously. 
The statement then follows since every $2\phi$-sector $\Phi$ has to contain 
  entirely a $\phi$-sector $\Phi_i$, and by a union bound over all events.
\end{proof}
We are now ready for the upper bound on the second largest component.
\begin{proposition}\label{p:UpperBound}
Let $G=(V,E)$ be chosen according to $\poimod_{\alpha,\nu}(n)$. W.e.h.p., 
\[
L_2(G) = O(\log^{\frac{1}{1-\alpha}} n).
\]
\end{proposition}
\begin{proof}
Let $c>0$, $L:=L(c+1)$, $\ell:=\ell(L)$, $L':=L'(c+1)$, and $\phi:=\phi(L')$ be as in the statement of Lemma~\ref{lem:maxAngle}.
By a union bound and appropriate choices of $L$ and $L'$, Lemma~\ref{lem:HtoG} and Lemma~\ref{lem:maxAngle} imply that, 
  with probability at least $1-O(n^{-(c+1)})$, all vertices in $B_{O}(\ell)$
  belong to the center component and 
  every $2\phi$-sector contains at least one vertex $v\in B_{O}(\ell)$.
Then, every vertex $x$ outside the center component 
  belongs to $B_{O}(R)\setminus B_{O}(\ell)$. 
Now, consider a component $C$ distinct from the center component and 
  let $u,u'$ be vertices in $C$ such that
  $|\theta_{u'}-\theta_{u}|=\max_{x,x'} |\theta_x-\theta_{x'}|$, where the 
  maximum is taken over all pairs of vertices $x,x'$ belonging to $C$.
If we had $|\theta_{u'}-\theta_{u}| \ge 2\phi$,
  then by our conditioning 
  there would be a vertex $v\in B_{O}(\ell)$
  (thus in the center component) such that 
  $\theta_{u} \le \theta_v \le \theta_{u'}$.
Since there exists a path in $C$ between $u$ and $u'$ containing only vertices 
  $u_j$ with $r_{u_j} > \ell$, in such a path there must be a pair of 
  vertices, say $u_i, u_j$, with 
  $r_v \le r_{u_i}, r_{u_{j}}$, $u_iu_{j} \in E$, and 
  $\theta_{u_i} \le \theta_v \le \theta_{u_{j}}$.
  By Lemma~\ref{lem:FK15}, also $u_iv \in E$ and $u_{j}v \in E$, 
  and hence $u$ and $u'$ are connected to the center component. 
Therefore, by our conditioning we may assume that 
  $|\theta_{u'}-\theta_{u}| < 2\phi$. 
Note that conditioning on the distribution of vertices inside $B_{O}(\ell)$ does not change the distribution of vertices in 
  $B_{O}(R)\setminus B_{O}(\ell)$.
Hence, since $\phi=\omega(\log n/n)$, by Lemma~\ref{lem:Poisson}, w.e.h.p.~we get $|C|=O(\phi n)=O((\log n)^{\frac{1}{1-\alpha}})$. 
By a union bound over all events, 
  with probability at least $1-O(n^{-c})$, it holds that connected components distinct from the center 
  component are of size $O((\log n)^{\frac{1}{1-\alpha}})$, and the statement follows from Proposition~\ref{prop:bfm}.
\end{proof}

\subsection{Lower bound}\label{subs:Lower}
We next turn to prove a lower bound matching the bound of Proposition~\ref{p:UpperBound}. Let $M=:M(\alpha, \nu)$ throughout this subsection be a sufficiently large constant. Partition $B_O(R)$ into $\psi$-sectors with $\psi:=(\nu/n)^{1-\beta}$ for a sufficiently small constant $\beta:=\beta(\alpha, M, \nu)$ (first, $M$ has to be chosen sufficiently large as a function of the model parameters $\alpha$ and $\nu$, independent of $\beta$, and then, $\beta$ has to be chosen small enough). 
Fix throughout this subsection 
  $\ell:=R-\frac{\log R}{1-\alpha}+\frac{M}{1-\alpha}$ 
  (recall that we suppose that $\ell$ is an integer). 
Let $\phi := 9\theta_{R}(\ell,\ell)$.
By Lemma~\ref{lem:aproxAngle} and Remark~\ref{rem:aproxAngle} thereafter,
and since $R=2\log\frac{n}{\nu}$,
\begin{equation}\label{eqn:phi}
\theta_{R}(\ell,\ell) = (2+o(1))\frac{\nu}{n}e^{R-\ell}
  = (2+o(1))\frac{\nu}{n}R^{\frac{1}{1-\alpha}}e^{-\frac{M}{1-\alpha}}.
\end{equation}
For each $\psi$-sector $\Psi$, consider the region 
  $\Upsilon_{\ell}:=\Upsilon_{\ell}(\Psi)$ consisting of those points  
  of $B_{O}(\ell)\setminus B_{O}(\ell-1)$ that belong to the 
  $\phi$-sector having the same bisector as $\Psi$. Formally, defining $\Phi_{\Psi}$ as the $\phi$-sector having the same bisector as $\Psi$, we have $\Upsilon_{\ell}=\Phi_{\Psi}\cap B_{O}(\ell)\setminus B_{O}(\ell-1)$. Next, we establish a lower bound on the probability that  
  $V\cap\Upsilon_{\ell}$ induces a connected component of $G$. 
Actually, we establish a stronger fact.
In the ensuing discussion, unless we say otherwise, the $\psi$-sector
  $\Psi$ is assumed to be given and all regions as well as subgraphs
  mentioned depend on $\Psi$.
\begin{lemma}\label{ev:B}
Let $\Upsilon'_{1},\ldots,\Upsilon'_{18}$ be a partition
  of $\Upsilon_{\ell}$ into $18$ parts, 
  each $\Upsilon'_{i}$
  obtained as the intersection of $\Upsilon_{\ell}$ and a 
  $\frac{\phi}{18}$-sector.
The following hold:
\begin{enumerate}[(i).-]
\item\label{it:B1}
Let $\calB$ be the event that $V\cap\Upsilon'_{i}$
  is non-empty for every $i=1,...,18$.
Then, $\calB$ occurs a.a.s.

\item\label{it:B2}
For sufficiently large $n$, all vertices in 
  $V\cap\Upsilon_{\ell}$ belong to the same connected component.
\end{enumerate}
\end{lemma}
\begin{proof}
To prove~\eqref{it:B1}, observe that by our choice of $\phi$,
  Corollary~\ref{cor:medidalb}, and~\eqref{eqn:phi}, for each $i$,
\begin{align*}
\mu(\Upsilon'_{i}) 
  & = \frac12\theta_{R}(\ell,\ell)\mu(B_{O}(\ell)\setminus B_{O}(\ell-1))
  = (1+o(1))(1-e^{-\alpha})\frac{\nu}{n}e^{(1-\alpha)(R-\ell)}.
\end{align*}
Clearly, the events $V\cap\Upsilon'_{1}\neq\emptyset,...,
  V\cap\Upsilon'_{18}\neq\emptyset$ are independent. 
Hence, by our choice of $\ell$, 
\begin{equation*}\label{evB}
\Pr{(\mathcal{B})} 
  =(1-e^{-\nu(1+o(1))(1-e^{-\alpha})e^{(1-\alpha)(R-\ell)}})^{18}=1+o(1).
\end{equation*}
To prove~\eqref{it:B2}, note that (by Remark~\ref{rem:monotonTheta}) two vertices in $B_{O}(\ell)\setminus B_{O}(\ell-1)$ (and thus the same holds for vertices in $\Upsilon_{\ell}\cap B_{O}(\ell)\setminus B_{O}(\ell-1)$) are neighbors if they form an 
  angle at the origin of at most $\theta_{R}(\ell,\ell)$.
Thus, every vertex in $\Upsilon'_{i}$ is connected 
  by an edge to every vertex in 
  $\Upsilon'_{i-1}\cup\Upsilon'_{i}\cup\Upsilon'_{i+1}$, 
  since the maximal angle such pairs of vertices form is,
  by our choice of $\phi$, at most 
  $2\frac{\phi}{18}=\theta_{R}(\ell,\ell)$. Because of~\eqref{it:B1}, we get that all vertices in
     $V\cap\Upsilon_{\ell}$ must be connected.
\end{proof}

Henceforth, for two points $p,p'\in B_{O}(R)$ let $\Delta\phi_{p,p'}$ denote 
  the (smaller) angle in $[0,\pi)$ between $p$ and $p'$ formed at the origin, i.e., $\Delta\phi_{p,p'}:=\min\{|\theta_p-\theta_{p'}|,|2\pi-\theta_p+\theta_{p'}|\}$. By definition of $\theta_R(\cdot,\cdot)$, we know that 
  $\dist(p,p')\leq R$ if and only if 
  $\Delta\phi_{p,p'}\leq\theta_R(r_p,r_{p'})$.
Now, for $i\in\{0,\ldots,R-\ell\}$, 
  let $\Upsilon_{\ell+i}$ be the collection of points in 
  $B_{O}(\ell+i)\setminus B_{O}(\ell+i-1)$ that belong to the $(2\upsilon_{\ell+i})$-sector
  with the same bisector as $\Psi$ where
\[
\upsilon_{\ell+i}:=\frac{\phi}{2}+\sum_{j=0}^{i-1}\theta_R(\ell-1+j,\ell+j).
\]
(Note that the preceding definition of $\Upsilon_{\ell}$ is consistent
  with the one given before Lemma~\ref{ev:B}.)

Similarly, 
  for $i\in\{0,\ldots,R-\ell\}$, 
  let $\Xi_{\ell+i}$ be the collection of points in 
  $B_{O}(\ell+i)\setminus B_{O}(\ell+i-1)$ that belong to the 
  $(2\xi_{\ell+i})$-sector
  with the same bisector as $\Psi$ where
\[
\xi_{\ell+i}:=\theta_{R}(\ell-1+i,\ell-1+i)+\frac{\phi}{2}+\xi,
\]
and $\displaystyle\xi:=\sum_{j=0}^{R-\ell-1}\theta_R(\ell-1+j,\ell+j)$.

Finally, let $\displaystyle\Xi
     := \bigcup_{i=0}^{R-\ell} \Xi_{\ell+i}$
  and $\displaystyle\Upsilon
     := \bigcup_{i=0}^{R-\ell}\Upsilon_{\ell+i}$
  (see Figure~\ref{fig:Xi}).
Clearly, $\Upsilon\subseteq\Xi$.

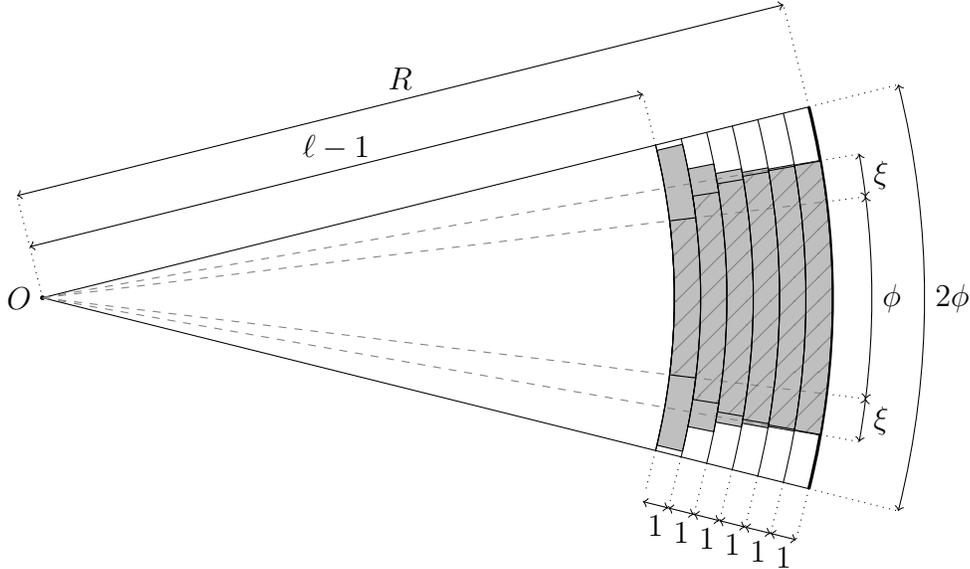
\begin{figure}[htbp]
\centering
\begin{tikzpicture}[scale=0.7]
\pgfmathsetmacro{\angXi}{10};
\pgfmathsetmacro{\angIn}{7};
\pgfmathsetmacro{\angOut}{14};
\pgfmathsetmacro{\rout}{15};
\pgfmathsetmacro{\rin}{12};
\node [left] at (0,0) {$O$};
\draw [fill] (0,0) circle (1 pt);

\draw (0,0) -- (-\angOut:\rout);
\draw (0,0) -- (\angOut:\rout);
\draw [very thick] (0,0) ++(-\angOut:\rout) arc (-\angOut:\angOut:\rout);
\draw (0,0) ++(-\angOut:\rin) arc (-\angOut:\angOut:\rin);

\draw [fill=gray!50] (0,0) ++(-13.5:\rin) arc (-13.5:13.5:\rin) -- (13.5:\rin+0.5) arc (13.5:-13.5:\rin+0.5) -- cycle;
\draw [fill=gray!50] (0,0) ++(-11.29:\rin+0.5) arc (-11.29:11.29:\rin+0.5) -- (11.29:\rin+1.0) arc (11.29:-11.29:\rin+1.0) -- cycle;
\draw [fill=gray!50] (0,0) ++(-10.47:\rin+1.0) arc (-10.47:10.47:\rin+1.0) -- (10.47:\rin+1.5) arc (10.47:-10.47:\rin+1.5) -- cycle;
\draw [fill=gray!50] (0,0) ++(-10.17:\rin+1.5) arc (-10.17:10.17:\rin+1.5) -- (10.17:\rin+2.0) arc (10.17:-10.17:\rin+2.0) -- cycle;
\draw [fill=gray!50] (0,0) ++(-10.06:\rin+2.0) arc (-10.06:10.06:\rin+2.0) -- (10.06:\rin+2.5) arc (10.06:-10.06:\rin+2.5) -- cycle;
\draw [fill=gray!50] (0,0) ++(-10.02:\rin+2.5) arc (-10.02:10.02:\rin+2.5) -- (10.02:\rin+3.0) arc (10.02:-10.02:\rin+3.0) -- cycle;

\draw [dashed,gray] (0,0) -- (-\angXi:\rout);
\draw [dashed,gray] (0,0) -- (\angXi:\rout);
\draw [dashed,gray] (0,0) -- (-\angIn:\rout);
\draw [dashed,gray] (0,0) -- (\angIn:\rout);

\draw [dotted] (-\angXi:\rout) -- (-\angXi:\rout+0.75);
\draw [dotted] (\angXi:\rout) -- (\angXi:\rout+0.75);
\draw [dotted] (-\angIn:\rout) -- (-\angIn:\rout+0.75);
\draw [dotted] (\angIn:\rout) -- (\angIn:\rout+0.75);
\draw [dotted] (-\angOut:\rout) -- (-\angOut:\rout+1.75);
\draw [dotted] (\angOut:\rout) -- (\angOut:\rout+1.75);
\draw [<->] (0,0) ++(-\angIn:\rout+0.75) arc (-\angIn:\angIn:\rout+0.75) node [midway,right] {$\phi$};
\draw [<->] (0,0) ++(-\angOut:\rout+1.75) arc (-\angOut:\angOut:\rout+1.75) node [midway,right] {$2\phi$};
\draw [<->] (0,0) ++(-\angIn:\rout+0.75) arc (-\angIn:-\angXi:\rout+0.75) node [midway,right] {$\xi$};
\draw [<->] (0,0) ++(\angIn:\rout+0.75) arc (\angIn:\angXi:\rout+0.75) node [midway,right] {$\xi$};

\draw [dotted] (\angOut+90:0) -- (\angOut+90:2.0);
\draw [dotted] (\angOut:\rout) -- ++(\angOut+90:2.0);
\draw [dotted] (\angOut:\rin) -- ++(\angOut+90:1.0);
\draw [<->] (\angOut+90:1.0) -- ++(\angOut:\rin) node [midway,above] {$\ell-1$};
\draw [<->] (\angOut+90:2.0) -- ++(\angOut:\rout) node [midway,above] {$R$};
\draw [dotted] (-\angOut:\rin) -- ++(-\angOut-90:0.75);
\draw [<->] (-\angOut-90:1.0) ++(-\angOut:\rin) -- ++(-\angOut:0.5) node [midway,below] {$1$};
\draw [<->] (-\angOut-90:1.0) ++(-\angOut:\rin+0.5) -- ++(-\angOut:0.5) node [midway,below] {$1$};
\draw [<->] (-\angOut-90:1.0) ++(-\angOut:\rin+1.0) -- ++(-\angOut:0.5) node [midway,below] {$1$};
\draw [<->] (-\angOut-90:1.0) ++(-\angOut:\rin+1.5) -- ++(-\angOut:0.5) node [midway,below] {$1$};
\draw [<->] (-\angOut-90:1.0) ++(-\angOut:\rin+2.0) -- ++(-\angOut:0.5) node [midway,below] {$1$};
\draw [<->] (-\angOut-90:1.0) ++(-\angOut:\rin+2.5) -- ++(-\angOut:0.5) node [midway,below] {$1$};
\draw [dotted] (-\angOut:\rin+0.5) -- ++(-\angOut-90:1.0);
\draw [dotted] (-\angOut:\rin+1.0) -- ++(-\angOut-90:1.0);
\draw [dotted] (-\angOut:\rin+1.5) -- ++(-\angOut-90:1.0);
\draw [dotted] (-\angOut:\rin+2.0) -- ++(-\angOut-90:1.0);
\draw [dotted] (-\angOut:\rin+2.5) -- ++(-\angOut-90:1.0);
\draw [dotted] (-\angOut:\rin+3.0) -- ++(-\angOut-90:1.0);

\pgfdeclarepatternformonly{ltrdiagonals}%
{\pgfqpoint{-1pt}{-1pt}}{\pgfqpoint{10pt}{10pt}}%
{\pgfqpoint{9pt}{9pt}}{\pgfsetlinewidth{0.4pt}\pgfpathmoveto{\pgfqpoint{0pt}{0pt}}%
\pgfpathlineto{\pgfqpoint{9.1pt}{9.1pt}}%
\pgfusepath{stroke}}%
\draw[
      fill opacity=0.64,
      postaction={fill,pattern=ltrdiagonals,pattern color=black!84}
     ] (0,0) ++(-\angIn:\rin) arc (-\angIn:\angIn:\rin) -- (\angIn:\rin+0.5) arc (\angIn:8.89:\rin+0.5) -- (8.89:\rin+1.0) arc (8.89:9.6:\rin+1.0) -- (9.6:\rin+1.5) arc (9.6:9.85:\rin+1.5) -- (9.85:\rin+2.0) arc (9.85:9.94:\rin+2.0) -- (9.94:\rin+2.5) arc (9.94:9.98:\rin+2.5) -- (9.98:\rout) arc (9.98:-9.98:\rout) -- (-9.98:\rin+2.5) arc (-9.98:-9.94:\rin+2.5) -- (-9.94:\rin+2.0) arc (-9.94:-9.85:\rin+2.0) -- (-9.85:\rin+1.5) arc (-9.85:-9.6:\rin+1.5) -- (-9.6:\rin+1.0) arc (-9.6:-8.89:\rin+1.0) -- (-8.89:\rin+0.5) arc (-8.89:-\angIn:\rin+0.5) -- cycle;

\draw (0,0) ++(-\angOut:\rin+0.0) arc (-\angOut:\angOut:\rin+0.0);
\draw (0,0) ++(-\angOut:\rin+0.5) arc (-\angOut:\angOut:\rin+0.5);
\draw (0,0) ++(-\angOut:\rin+1.0) arc (-\angOut:\angOut:\rin+1.0);
\draw (0,0) ++(-\angOut:\rin+1.5) arc (-\angOut:\angOut:\rin+1.5);
\draw (0,0) ++(-\angOut:\rin+2.0) arc (-\angOut:\angOut:\rin+2.0);
\draw (0,0) ++(-\angOut:\rin+2.5) arc (-\angOut:\angOut:\rin+2.5);
\end{tikzpicture}
\caption{Region $\Xi$ is shown shaded in gray and
  region $\Upsilon$ diagonally hatched (not to scale). The two shaded non-hatched regions correspond to walls provided they do not contain vertices.}\label{fig:Xi}
\end{figure}
Note that $\Xi\setminus\Upsilon$ is comprised of two similar
  connected geometric regions of $B_{O}(R)$.
  Denote by $\calW$ either one of them.
  We say that a region $\calW'\subseteq B_{O}(R)$ that is obtained by a rotation
  around the origin of $\calW$ is a \emph{wall} if it does not contain
  any element of $V$, i.e., $V\cap\calW'=\emptyset$. 

Next, we establish several facts concerning regions $\Xi$ and $\Upsilon$,
  but first we bound $\xi$ just defined.
By Lemma~\ref{lem:aproxAngle}, 
  the formula for the sum of a geometric series,
  since $R=2\log\frac{n}{\nu}$, and by our choice of $\ell$ 
\begin{equation}\label{eqn:xi}
\xi 
  = (2+o(1))\frac{\nu}{n}e^{R-\ell+\frac12}\sum_{j=0}^{R-\ell-1}e^{-j}
  = (2+o(1))\frac{e^{3/2}}{e-1}\frac{\nu}{n}e^{R-\ell}.
\end{equation}
Since $e^{3/2}/(e-1)< 3$, by~\eqref{eqn:phi},
  and our choice of $\phi$, for sufficiently large $n$, 
\begin{equation}\label{eqn:phiXi}
\xi < 3\theta_R(\ell,\ell) = \frac{1}{3}\phi. 
\end{equation}

Let $\calC$ be the event that there is no vertex in $\Xi\setminus\Upsilon$.
The next lemma says that $\Xi\setminus\Upsilon$ contains, with a not too small probability,
  two disjoint walls (one in each
  side of the bisector of $\Upsilon$).
\begin{lemma}\label{ev:C}
For sufficiently large $n$, the following hold:
\begin{enumerate}[(i).-]
\item\label{it:C1}
For some constant $C_0=C_0(\alpha)$ depending only on $\alpha$, the probability that $\calC$ occurs is at least $e^{-C_{0}\nu Re^{-M}}$.

\item\label{it:C2}
If $p\in (B_{O}(R)\setminus B_{O}(\ell-1))\setminus \Xi$ 
  and $p'\in\Upsilon$, then 
  $\dist(p,p')> R$.
\end{enumerate}
\end{lemma}
\begin{proof}
To prove~\eqref{it:C1}, observe that by  Lemma~\ref{lem:aproxAngle}, Corollary~\ref{cor:medidalb},
  by definition of $\Xi_{\ell+i}$, the formula
  for the sum of a geometric series and since $\sqrt{e}/(e-1) < 1$,
\begin{align*}
& \mu(\Xi_{\ell+i}\setminus\Upsilon_{\ell+i}) 
  = 2(\xi_{\ell+i}-\upsilon_{\ell+i})
      \mu(B_{O}(\ell+i)\setminus B_{O}(\ell-1+i)) \\
  & \quad = 2
  \Big((2e+o(1))\frac{\nu}{n}e^{R-\ell-i}
  +(2\sqrt{e}+o(1))\frac{\nu}{n}e^{R-\ell}\sum_{j=i}^{R-\ell-1}e^{-j}\Big)
  (1-e^{-\alpha})e^{-\alpha(R-\ell-i)} \\
  & \quad \leq  (8e+o(1))(1-e^{-\alpha})\frac{\nu}{n}e^{(1-\alpha)(R-\ell-i)}.
\end{align*}
Hence, again by the formula for the sum of a geometric series and our choice of $\ell$,
\[
\mu(\Xi\setminus\Upsilon)\leq \frac{(8e+o(1))(1-e^{-\alpha})}{1-e^{-(1-\alpha)}}\frac{\nu}{n} e^{(1-\alpha)(R-\ell)}
  (1-e^{-(1-\alpha)(R-\ell+1)})
  <  C_0\frac{\nu}{n}Re^{-M},
\]
where $C_0$ is a constant depending only on $\alpha$. The sought after lower bound on the probability that 
  $V\cap\Xi\setminus\Upsilon$ is empty follows immediately.
  
Next, consider~\eqref{it:C2}.
To prove that $\dist(p,p')>R$ it suffices to show that 
  $\Delta\phi_{p,p'}>\theta_{R}(r_p,r_{p'})$.
Assume $p\in (B_{O}(\ell+i)\setminus B_{O}(\ell+i-1))\setminus\Xi$
  and $p'\in \Upsilon_{\ell+i'}$.
Hence, 
\begin{align*}
\Delta\phi_{p,p'} & >\xi_{\ell+i}-\upsilon_{\ell+i'}
  = \theta_{R}(\ell-1+i,\ell-1+i)+\sum_{j=i'}^{R-\ell-1}\theta_R(\ell-1+j,\ell+j)
\\
  & 
  = (2e+o(1))\frac{\nu}{n}e^{R-\ell}\big(e^{-i}
   + e^{-i'}\frac{\sqrt{e}}{e-1}(1-e^{-(R-\ell-i')})\big),
\end{align*}
where the last equality follows from Lemma~\ref{lem:aproxAngle},
  since $R=2\log\frac{n}{\nu}$, and the formula for the sum of a geometric 
  series.
If $i' < R-\ell$, then $e^{-(R-\ell-i')} \leq e^{-1}$, and 
  since $\sqrt{e}(1-e^{-1})/(e-1)\approx 0.61$, 
  applying Jensen's inequality we obtain that 
  for sufficiently large $n$, 
\[
\Delta\phi_{p,p'} 
  > (2e+o(1))\frac{\nu}{n}e^{R-\ell}e^{-\frac12(i+i')}= \theta_{R}(\ell-1+i,\ell-1+i').
\]
If $i'=R-\ell$, then $e^{-i} \ge e^{-\frac12(i+i')}$ and $e^{-(R-\ell-i')}=1$, 
  so by Remark~\ref{rem:monotonTheta},
\[
\Delta\phi_{p,p'} >  \theta_{R}(\ell-1+i,\ell-1+i) \ge \theta_{R}(\ell-1+i, \ell-1+i').
\] 

Now, by Remark~\ref{rem:monotonTheta}, Lemma~\ref{lem:aproxAngle},
and again since $R=2\log\frac{n}{\nu}$,
\[
\theta_{R}(r_p,r_{p'})\leq\theta_{R}(\ell-1+i,\ell-1+i').
\]
The last three displayed bounds imply that,
  for a sufficiently large $n$ (independent of $i$ and $i'$), we have 
  $\Delta\phi_{p,p'}> \theta_{R}(r_p,r_{p'})$ as claimed.  
\end{proof}

We stress that Lemma~\ref{ev:C} part (\ref{it:C2}) corresponds exactly to the second property satisfied by walls as described in Section~\ref{sec:intro}.

For a given $\Psi$, let $H$ be the subgraph of $G$ induced by 
  $V\cap\Upsilon$, where $\Upsilon=\Upsilon(\Psi)$, 
  and denote by $C(\Upsilon)$ the collection of vertices of the connected components 
  of $H$ that contain at least one vertex in $V\cap\Upsilon_\ell$. 

Let $\calG$ be the event that $|C(\Upsilon)|=\Omega((\log n)^{\frac{1}{1-\alpha}})$.
Thus, by definition, $\calG$ depends only on what happens inside
  $\Upsilon$.
\begin{lemma}\label{ev:G}
The event $\calG$ occurs a.a.s.
\end{lemma}
\begin{proof}
Let $\eta=\eta(\alpha, \nu)$ be a sufficiently large constant, let 
  $\Phi$ be the $\frac{\phi}{3}$-sector 
  with the same 
  bisector as $\Psi$, and let $\ell':=R-\frac{c\log R}{1-\alpha}$ for some 
  small constant $0 < c < 1$. 
For each vertex $z \in V_{R-\eta}\cap\Phi$, 
let $X_z$ be the indicator random variable for the event that there is a path 
  $z=z_{R-\eta},\ldots, z_{\ell'}$ in $G$ so that $z_i\in V_{i}:=V\cap B_{O}(i)\setminus B_{O}(i-1)$
  for every $i$.

We claim that for a sufficiently large $n$, there is a $\delta>0$ such that if
  $z\in V_{R-\eta}\cap\Phi$, then the expected value of $X_z$ is 
  at least $\delta$.
Indeed, suppose that for some $i$ we found a path until $z_{i+1}$. 
By Lemma~\ref{lem:aproxAngle}, Remark~\ref{rem:monotonTheta},
  and Corollary~\ref{cor:medidalb}, 
  the region $\calR\subseteq B_{O}(i)\setminus B_{O}(i-1)$ 
  in which the next  
  vertex $z_{i}$ with the desired properties can be found satisfies
\begin{equation}\label{measureB}
\mu(\calR) 
  \ge \theta_R(i+1,i)\mu(B_{O}(i)\setminus B_{O}(i-1))
  = (2+o(1))(1-e^{-\alpha})\frac{\nu}{n}e^{(1-\alpha)(R-i)-\frac12},
\end{equation}
and hence, with probability at most 
  $e^{-(2+o(1))\nu (1-e^{-\alpha})e^{(1-\alpha)(R-i)-\frac12}}$ no such vertex is found.
Thus, for some positive constant $\delta > 0$, assuming 
  $\eta$ was chosen sufficiently large (and also $n$ sufficiently large),
\begin{equation}\label{ev:Pz}
\Ex{X_z}
  \ge 1-\sum_{i=\ell'}^{R-\eta-1} e^{-(2+o(1))\nu (1-e^{-\alpha}) e^{(1-\alpha)(R-i)-\frac12}} 
  \ge \delta.
\end{equation}

Now, let $X:=\sum_{z}X_z$ where the summation is over the $z$'s in
  $V_{R-\eta}\cap\Phi$.
We claim that $X=(1+o(1))\Ex{X}$ a.a.s.
Indeed, by Lemma~\ref{lem:aproxAngle}, Corollary~\ref{cor:medidalb},
  and~\eqref{eqn:phi}, 
  we have $\mu(\Phi\cap B_{O}(R-\eta)\setminus B_{O}(R-\eta-1))
  =\Theta(\frac{1}{n}R^{\frac{1}{1-\alpha}})$.
Thus, by Lemma~\ref{lem:Poisson}, for $\eta$ large enough, w.e.h.p., $|V_{R-\eta}\cap\Phi|
  =\Theta((\log n)^{\frac{1}{1-\alpha}})$, and hence  by~\eqref{ev:Pz}, 
  $\Ex{X}=\Theta((\log n)^{\frac{1}{1-\alpha}})$. 
Moreover, in case there is a path 
  $z=z_{R-\eta},\ldots, z_{\ell'}$ in $G$ so that $z_i\in V_{i}$
  for every $i$, the total angle between 
  $z$ and $z_{\ell'}$ is 
\[
\Delta\phi_{z,z_{\ell'}}
  \leq\sum_{i=\ell'}^{R-\eta-1}\Delta\phi_{z_i,z_{i+1}}
  \leq\sum_{i=\ell'}^{R-\eta-1}\theta_{R}(i-1,i)
  = O\Big(\frac{\nu}{n}e^{R-\ell'}\Big)
  = O\Big(\frac{\nu}{n}R^{\frac{c}{1-\alpha}}\Big)=o(\phi).
\] 
Hence, if two such vertices $z, z'\in V_{R-\eta}\cap\Phi$ are at an 
  angle $\omega\big(\frac{1}{n}(\log n)^{\frac{c}{1-\alpha}}\big)$, then 
  $X_z$ and $X_{z'}$ are independent. 
Since $c < 1$, most pairs of vertices are at angular distance 
  $\omega\big(\frac{1}{n}(\log n)^{\frac{c}{1-\alpha}}\big)$, and thus 
  $\Ex{(X^2)}=(1+o(1))(\Ex{X})^2$, so by Chebyshev's 
  inequality, a.a.s.~$X=(1+o(1))\Ex{X}$ as claimed.

By the preceding discussion, in order to conclude
  that a.a.s.~$|C(\Upsilon)|=(1+o(1))\Ex{X}=  
  \Omega((\log n)^{\frac{1}{1-\alpha}})$ it is enough 
  to show that a.a.s.~the
  following event occurs: for every vertex $z$ in 
  $V_{\ell'}\cap\Phi$ there exists a path 
  $z=z_{\ell'}\ldots z_{\ell}$ in $G$ with $z_i\in V_{i}$. This fact follows observing that similar calculations as the ones performed above
  to estimate $|V_{R-\eta}\cap\Phi|$ yield that w.e.h.p.~$|V_{\ell'}\cap\Phi|
  =O((\log n)^{\frac{1}{1-\alpha}})$.
By calculations as in~\eqref{measureB} together with a union bound, 
  the desired event does not occur with probability
\begin{align*}
O((\log n)^{\frac{1}{1-\alpha}}e^{-\log^c n})
+\Pr{\big(|V_{\ell'}\cap\Phi| = \omega((\log n)^{\frac{1}{1-\alpha}})\big)}
  & = e^{\Theta(\log \log n)-\log^c n}+o(n^{-1}) \\ & =e^{-\Theta(\log^c n)}.
\end{align*}

Finally, let $z$ be a vertex in $V_{R-\eta}\cap\Phi$ for which
  there exists a path $z=z_{R-\eta},\ldots,z_{\ell}$ in 
  $G$ with $z_i\in V_i$ for all $i$.
Note that the angle $\Delta\phi_{z,z_{\ell}}$
  between the endvertices $z$ and $z_{\ell}$ 
  of the path satisfies, by Remark~\ref{rem:monotonTheta}, 
\[
\Delta\phi_{z,z_{\ell}}
  \leq\sum_{i=\ell}^{R-\eta-1}\Delta\phi_{z_i,z_{i+1}}
  \leq\sum_{i=\ell}^{R-\eta-1}\theta_{R}(i-1,i)
  \leq\sum_{i=0}^{R-\ell-1}\theta_{R}(\ell-1+i,\ell+i) = \xi.
\]
Thus, by~\eqref{eqn:phiXi}, the total angle between 
$z$ and $z_{\ell}$ is at most $\frac{1}{3}\phi$.
Since $z$ is a vertex in $\Phi$, it lies within an angle of 
  at most $\frac{\phi}{6}$
  of the bisector of $\Psi$.
  Thus, all vertices of the $z,...,z_{\ell}$ path are
  within a $\phi$-sector with the same bisector as $\Upsilon$ 
  so by construction are also within $\Upsilon$, and hence in 
  establishing that $\calG$ occurs a.a.s.~only 
  $\Upsilon\cap\Phi$ needs to be exposed.
\end{proof}

Now, in order to have a component disconnected from 
  the giant component it is enough that all vertices in $\Upsilon$
  have no neighbors in $B_{O}(R)\setminus\Upsilon$. For vertices in $\Upsilon$ not to have neighbors in
  $(B_{O}(R)\setminus B_{O}(\ell-1))\setminus\Upsilon$,
    by Lemma~\ref{ev:C} Part~\eqref{it:C2}, it is
    enough that $V\cap\Xi\setminus\Upsilon$ is empty, as no vertex in
    $\Upsilon$ can have a neighbor in
    $(B_{O}(R)\setminus B_{O}(\ell-1))\setminus\Xi$. 
However, vertices in $\Upsilon$ could have
  neighbors in $B_{O}(\ell-1)$.
We next deal with this situation.
First, we show that it is unlikely for such neighbors to fall within
  $B_{O}(\ell-1)\setminus B_{O}((1-\frac{\beta}{2})R)$ and then we deal
  with the possibility of having neighbors in 
  $B_{O}((1-\frac{\beta}{2})R)$
  (recall that $\beta=\beta(M)$ is a sufficiently small constant).

Let $\calH$ be the event that no vertex in $B_{O}(\ell-1)\setminus B_{O}((1-\frac{\beta}{2})R)$ is within distance $R$ of $\Upsilon$.
\begin{lemma}\label{ev:H}
There is a constant $C_1=C_1(\alpha)$ depending only on $\alpha$
  so that for sufficiently large $n$ the event $\calH$ occurs with
   probability at least $e^{-C_1\nu Re^{-M}}$.
Moreover, all area exposed in $\calH$ is inside 
  $\Psi\cap B_{O}(\ell-1)\setminus B_{O}((1-\frac{\beta}{2})R)$.
\end{lemma}
\begin{proof}
Since by definition $\upsilon_{\ell+i}$ increases with $i$, 
  all points in $\Upsilon$ are within an angle 
  $2\upsilon_{R}= 2(\frac{\phi}{2}+\xi)$, so recalling~\eqref{eqn:phiXi}
  also within an angle $2\phi$.
Moreover, by Remark~\ref{rem:monotonTheta}, 
  between two points within distance at most $R$ one of which is 
  in $B_O(j+1)\setminus B_{O}(j)$, $(1-\frac{\beta}{2})R\leq j\leq\ell-2$,
  and the other one in $\Upsilon$ there is an angle at the origin
  of at most $\theta_R(j,\ell-1)$.
Hence, by Lemma~\ref{lem:aproxAngle} and Lemma~\ref{lem:muBall}, and again
  by our choices for $\phi$ and $\ell$, 
  the expected number of neighbors of the vertices in $\Upsilon$ 
  that are inside $B_{O}(\ell-1)\setminus B_{O}((1-\frac{\beta}{2})R)$
  is at most

\begin{align*}
&
n\sum_{j=(1-\frac{\beta}{2})R}^{\ell-2}2(\upsilon_R+\theta_{R}(j,\ell-1))\mu(B_{O}(j+1)\setminus B_{O}(j)) \\
&\qquad \leq 
  2\phi n\mu(B_{O}(\ell))+ 2\sum_{j=(1-\frac{\beta}{2})R}^{\ell-2}\theta_{R}(j,\ell-1)n\mu(B_{O}(j+1))
\\
&\qquad \leq
  18(2+o(1))\nu e^{(1-\alpha)(R-\ell)} + 
  2(2e^{3/2-\alpha}+o(1))\nu e^{\frac12(R-\ell)}\sum_{k\geq R-\ell}e^{-(\alpha-\frac12)k}
  \\
&\qquad \leq 
  C_1\nu Re^{-M},
\end{align*}
where $C_1$ is a constant depending on $\alpha$, but independent of $M$.
The lower bound on $\PP(\calH)$ immediately follows.

To conclude, observe that 
  all area exposed in $\calH$ is inside the $\psi$-sector $\Psi$, as 
  all area exposed lies within an angle of 
  at most $2(\upsilon_{R}+\theta_{R}(\ell-1,(1-\frac{\beta}{2})R))$,
  which by the preceding discussion, Lemma~\ref{lem:aproxAngle}, 
  and our choices of $\psi$, $\phi$, and $\ell$, is at most 
\[
2\phi + 2(2\sqrt{e}+o(1))e^{\frac12(R-\ell-(1-\frac{\beta}{2})R)}
  = \Big(\frac{\nu}{n}\Big)^{1-\frac{\beta}{2}+o(1)}=o(\psi).
\]
\end{proof}
 
If for a sector $\Psi$ the events $\calB, \calC, 
  \calG, \calH$ hold, then we have found a 
  \emph{precomponent} of size $\Theta((\log n)^{\frac{1}{1-\alpha}})$: by 
  $\calB$ and $\calG$, there is a collection of vertices in $\Upsilon$ connected to each other
  (but perhaps not separated from the giant component) of size  $\Theta((\log n)^{\frac{1}{1-\alpha}})$.
All events are independent or positively correlated: $\calB$ and $\calG$ only depend on what 
  happens inside $\Upsilon$, and $\calC$ and $\calH$ depend on 
  what happens in disjoint regions outside $\Upsilon$, so
  $\calB\cap\calG$, $\calC$, and $\calH$ are independent. Moreover, events $\calB$ and $\calG$ are positively correlated.
Hence, by combining Lemmata~\ref{ev:B}, \ref{ev:C}, \ref{ev:G} and \ref{ev:H}
we get 
\begin{equation}\label{prec}
\Pr{(\calB\cap\calC\cap\calG\cap\calH)} 
  \ge  (1+o(1))e^{-C_0\nu Re^{-M}}e^{-C_1\nu Re^{-M}}
  = e^{-c_M R}
\end{equation}
for some constant $c_M=c_M(\alpha,\nu) > 0$ that can be made as small as 
  desired by choosing $M$ sufficiently large. 
Hence, for a given sector $\Psi$, the probability to have a precomponent 
  of size $\Theta((\log n)^{\frac{1}{1-\alpha}})$ is at least $e^{-c_M R}$, 
  independent of $\beta$. 
Observe also that all events $\calB$, $\calC$, $\calG$, 
  $\cal{H}$ expose only areas inside 
  $\Psi\setminus B_{O}((1-\frac{\beta}{2})R)$,  and thus the 
  events corresponding to the existence of a precomponent in disjoint
  $\psi$-sectors are independent.

Now, consider the partition of $B_O(R)$ into $\psi$-sectors 
  $\Psi_1,\ldots, \Psi_{2\pi/\psi}$. 
  By~\eqref{prec}, the probability that there is no $\psi_i$ with a 
  precomponent is therefore at most 
\begin{equation}\label{eq:Prec}
(1-e^{-c_M R})^{n^{1-\beta+o(1)}} \le e^{-n^{1-\beta-2c_M+o(1)}},
\end{equation}
which tends to $0$ faster than the inverse of any fixed polynomial in $n$, if $c_M$ is chosen small enough so that $1-\beta-2c_M > 0$ (such a choice exists, since $c_M$ is independent of $\beta$). Hence, w.e.h.p.~there exists a $\psi$-sector $\Psi$ that contains a precomponent of size 
  $\Theta((\log n)^{\frac{1}{1-\alpha}})$.

Let $\calS$ be the event that a randomly chosen $\psi$-sector $\Psi$ 
  is such that there is no vertex in $B_O((1-\frac{\beta}{2})R)$ at distance
  $R$ from $\Upsilon(\Psi)$.
\begin{lemma}\label{ev:P}
The event $\calS$ holds a.a.s.
\end{lemma}
\begin{proof}
By Remark~\ref{rem:monotonTheta}, 
  points in $B_{O}(j)\setminus B_O(j-1)$, $j\leq (1-\frac{\beta}{2})R$, at 
  distance at most $R$ from some point in $\Upsilon=\Upsilon(\Psi)$ 
  lie in a sector of angle at most 
  $2\theta_{R}(\ell-1,j-1)=2(2e+o(1))e^{\frac12(R-\ell-j)}$. Also, as observed at the beginning of the proof of Lemma~\ref{ev:H}, points inside $\Upsilon$ are within an angle of $2\phi$.
Hence, the region $\calR\subseteq
  B_{O}((1-\frac{\beta}{2})R)$ that needs to be empty in order for $\calS$ to 
  hold satisfies
\begin{align*}
\mu(\calR) & 
  = 2(2e+o(1))\sum_{j=0}^{(1-\frac{\beta}{2})R}\big(e^{\frac12(R-\ell-j)}+2\phi\big)e^{-\alpha(R-j)} \\
  & =O(R^{\frac{1}{1-\alpha}}e^{-\alpha R+(\alpha-\frac12)(1-\frac{\beta}{2})R})
  =n^{-1-\beta(\alpha-\frac12)+o(1)}.
  \end{align*}
Thus, the expected number of vertices inside $\calR$ is $o(1)$, and by 
  Markov's inequality, the event $\calS$ holds a.a.s. 
\end{proof}
To prove Theorem~\ref{thm:main}, observe now that if in addition to the existence of a precomponent the event $\calS$ holds, then the precomponent inside the randomly chosen $\psi$-sector $\Psi$ forms a connected component separated from the giant component. Since by~\eqref{eq:Prec} w.e.h.p.~there is a precomponent, by Lemma~\ref{ev:P}, by a union bound,  a.a.s.~there exists a component of size $\Theta((\log n)^{\frac{1}{1-\alpha}})$.
Summarizing, we have established the following:
\begin{proposition}\label{p:LowerBound}
For $\frac12 < \alpha < 1$, a.a.s.~$L_2(G)=\Theta((\log n)^{\frac{1}{1-\alpha}})$.
\end{proposition}
In fact, we have established that for some sufficiently small $\beta'>0$ a.a.s.~there are $\Omega(n^{\beta'})$ components of size $\Theta((\log n)^{\frac{1}{1-\alpha}})$: indeed, the partition of $B_O(R)$ into $\psi$-sectors can be grouped into groups of sectors making for a total angle of $n^{-\beta''}$, where $\beta'' > 0$ is chosen small enough so that~\eqref{eq:Prec} holds in each group, and also small enough, so that a union bound of all events over all groups still holds as well.

Proposition~\ref{p:UpperBound}, Proposition~\ref{p:LowerBound}, and the argument of the previous paragraph yield Theorem~\ref{thm:main}.
  

\section{Boundary cases of $\alpha$}\label{sec:boundary}
As noted in the introduction, for the hyperbolic random graph model, 
  the interesting range of the parameter is when $\frac12\leq\alpha\leq 1$.
In this section we investigate the size of the second largest component
  when $\alpha$ takes the values $\frac12$ or $1$.

\subsection{Case $\alpha=\frac12$}
By~\cite{BFM13b}, for $\alpha=\frac12$, it is known that for $\nu \ge \pi$, with probability tending to $1$,
the random graph $G$ is connected, whereas for smaller values of $\nu$, the probability of being connected is a continuous function of $\nu$ tending to $0$ as $\nu \to 0$.

On the one hand, for any constant $\nu$, there exists a constant $C$ (with $C$ being large as $\nu$ being small) so that a.a.s.~each vertex $v\in B_{O}(R-C)$ belongs to the giant component: 
indeed, for a vertex $v\in B_{O}(i)\setminus B_{O}(i-1)$ with 
  $\frac{R}{2} < i \le R-C$, the expected number of neighbors of 
  $v$ that belong to $B_{O}(j)\setminus B_{O}(j-1)$ with say $j>\frac{R}{2}$
  is $\Theta(e^{\frac12(R-i-j)}ne^{-\frac12(R-j)})=\Omega(1)$, where the constant 
  can be made large by making $C$ large. 
Hence, the probability that $v$ does not find 
  a neighbor in 
  $B_{O}(\frac{5R}{6})\setminus B_{O}(\frac{4R}{5})$ is $e^{-\Omega(R)}$, 
  where the constant in 
  the exponent can be made large by choosing $C$ large. 
By a similar argument, a.a.s.~every vertex in 
$B_{O}(\frac{5R}{6})\setminus B_{O}(\frac{4R}{5})$ 
  also has a neighbor in 
$B_{O}(\frac{R}{2})\setminus B_{O}(\frac{R}{4})$. 
Since all vertices within $B_{O}(\frac{R}{2})$ form a clique, all vertices 
  in $B_{O}(R-C)$ thus form a component of linear size.
Now, by choosing a sector $\Phi$ of angle $C' \log n /n$ with $C'=C'(C)$ sufficiently large, by standard estimates for Poisson random variables, each such sector will a.a.s.~contain a vertex in $B_{O}(R-C)$. Hence, a.a.s.~the second component 
  has to be contained in at most two consecutive sectors, as otherwise, by Lemma~\ref{lem:FK15}, any path whose vertices are all in $B_O(R)\setminus B_O(R-C)$ spanning two sectors, as well as the component to which such path belongs, would necessarily also have to be connected to a vertex of the giant component. 
Since the number of vertices in each sector of angle $C'\log n/n$ is a.a.s.~$O(\log n)$, this upper 
  bound holds also for the size of the second component.

On the other hand, for $\nu$ sufficiently small, we now show that with constant probability there exists a sector $\Phi$ of $B_O(R)$ of angle $\varepsilon \log n/n$ with $\varepsilon=\varepsilon(C)$ sufficiently small so that the following three events hold:
\begin{enumerate}[(i).-]
\item\label{it:boundary1} inside $\Phi$ there is no vertex $v$ in $B_{O}(R-C)$,
\item\label{it:boundary2} there exists a path of length $\varepsilon' \log n$ ($\varepsilon'$ sufficiently small) with all vertices being in $\varepsilon' \log n$ consecutive subsectors of $\Phi$ of angle $\varepsilon''/n$ 
  (with $\varepsilon''=\varepsilon''(C)$ small enough), 
  with all but the first and last vertex belonging to 
  $B_{O}(R-C_1+1)\setminus B_{O}(R-C_1)$ while the first and last belong 
  to $B_{O}(R-C_1)\setminus B_{O}(R-C_1-1)$ (for $C_1$ a small constant in comparison to $C$, but not too small so that any two vertices in consecutive 
  subsectors are adjacent; clearly, if a smaller value of $C_1$ is needed below, then this can be achieved 
  by making $\varepsilon''$ smaller), 
  except for this first and last vertex in all 
  these $\varepsilon' \log n$ subsectors there is no vertex in $B_{O}(R-C_1) \setminus B_{O}(R-C)$, 
  and there is no other vertex inside $B_{O}(R) \setminus B_{O}(R-C_1)$
  in the subsector of the first 
  and the last vertex,
\item\label{it:boundary3} no vertex of the path is connected to the giant component.
\end{enumerate} 
Note that for a fixed sector $\Phi$ condition~\eqref{it:boundary1} is satisfied with
  probability $e^{-\Theta(R)}$ with the constant in the exponent small
  for $\varepsilon$ small. 
  Condition~\eqref{it:boundary2}
  also holds with probability $e^{-\Theta(R)}$ with 
  the constant small for $\varepsilon'$ small. 
The last condition is satisfied if the 
  leftmost and rightmost vertex of the path do not connect to the giant 
  component: indeed, if a vertex outside the $\varepsilon' \log n$ subsectors containing the path is connected by an edge to a vertex that is neither the first nor the last vertex of the path, then by Lemma~\ref{lem:FK15} with $p'$ being the first (last) vertex on the path, $p''$ being another vertex on the path, and $p$ being the vertex outside, it must hold that $p$ is also connected by an edge to the first (last, respectively) vertex of the path. If a vertex inside $\Phi$ is connected by an edge to the giant component, then the vertex must be inside $B_{O}(R) \setminus B_{O}(R-C_1)$, since first there are no vertices in $\Phi\setminus B_{O}(R-C)$, and second, in all subsectors between the first and last vertex of the path there are no vertices in $B_{O}(R-C_1)$. Since any set of vertices connected to the giant component has to leave $\Phi$, once more by Lemma~\ref{lem:FK15},  at least one vertex of this set also has to be connected by an edge to the first (last, respectively) vertex of the path.
Hence, condition~\eqref{it:boundary3} again happens with probability $e^{-\Theta(R)}$ 
  (again with a constant in the exponent that can be made small for $\nu$ sufficiently small and $C_1$ still relatively small). 
The events described in the first two conditions are independent, and their intersection is positively correlated with the event described by the third condition.
Thus the expected number of sectors $\Phi$ for which all conditions hold 
  is $ne^{-\Theta(R)}/\log n=\omega(1)$ for 
  $\varepsilon, \varepsilon', \varepsilon''$ 
  sufficiently small. 
A second moment method analogous to the one in Lemma~\ref{ev:G} shows that different
  sectors are "almost" independent (special care is taken of
  vertices close to the center, that is, at a large constant distance, as in Lemma~\ref{ev:H} and Lemma~\ref{ev:P}). 
Thus, with constant probability such a sector exists (since there is constant probability that there is no vertex close to the center), and the second largest
  component is of size $\Omega(\log n)$, and 
  thus we obtain Proposition~\ref{p:alphaHalf}.

\subsection{Case $\alpha=1$}
Again by~\cite{BFM15}, for $\alpha=1$, 
  for $\nu$ sufficiently large, a.a.s.~there exists a giant component,
  whereas for $\nu$ small enough, a.a.s.~the largest component is sublinear. 
Choose $\ell:=\lambda R$ for some $0.51 < \lambda < 1$ ($\lambda$ depends on $\nu$ and has to be chosen closer to $1$ for $\nu$ larger) and consider 
  a vertex $v_1$ in $(B_{O}(\ell)\setminus B_{O}(\ell-1)) \cap \Upsilon_1$, where $\Upsilon_1$ is a sector containing all vertices $u$ with $\theta_u \in [0, \pm Cn^{-(\lambda-0.51)})$ for some large constant $C > 0$ (there are w.e.h.p.~$\Theta(ne^{-(R-\ell)}n^{-(\lambda-0.51)})=\Theta(n^{\lambda-0.49})$ such vertices, and hence w.e.h.p.~we find such a vertex $v_1$).  Clearly, the component of $v_1$ is at least the degree of $v_1$, which is w.e.h.p.~$\Theta(n^{1-\lambda})$. We will show that a.a.s.~there are polynomially many sectors like $\Upsilon_1$ containing a vertex of degree $\Theta(n^{1-\lambda})$ having all vertices of its component inside a sector whose angle is three times the angle of $\Upsilon_1$. 
    
    First, by standard estimates for Poisson random variables, a.a.s.~there is no vertex in $B_{O}(0.49R)$.
    Now, we try to construct a \emph{staircase} around the component of $v_1$ (a curve essentially like the boundary of the hatched region of Figure~\ref{fig:Xi} but with $\phi=0$). The \emph{left border} of the staircase has the following anchor points:
    $(\theta^1_{\ell-1}, r_{\ell-1})$, $(\theta^2_{\ell-1}, r_{\ell-1})$, where $\theta^1_{\ell-1}=\theta_{v_1}$, $r_{\ell-1}=\ell-1$ 
    and $\theta^2_{\ell-1}$ is chosen so that the point  $(\theta^2_{\ell-1}, r_{\ell})$ with $r_{\ell}=\ell$ is exactly at hyperbolic distance $R$ from $(\theta^1_{\ell-1}, r_{\ell-1})$ (and to the left of $v_1$, that is, its angular coordinate precedes in counterclockwise order $\theta_{v_1}$). Then, iteratively having found the two anchor points $(\theta^1_{\ell'}, r_{\ell'})$, $(\theta^2_{\ell'}, r_{\ell'})$ for some $\ell-1 \le \ell' \le R-2$, define the new anchor points corresponding to layer $\ell'+1$ as $(\theta^1_{\ell'+1}, r_{\ell'+1})$ and $(\theta^2_{\ell'+1}, r_{\ell'+1})$ with $\theta^1_{\ell'+1} = \theta^2_{\ell'}$, $r_{\ell'+1}=\ell'+1$ and $\theta^2_{\ell'+1}$ chosen so that the point $(\theta^2_{\ell'+1}, r_{\ell'+1})$ is exactly at hyperbolic distance $R$ from $(\theta^2_{\ell'}, r_{\ell'})$ (and to the left of it). For each anchor point $p=(\theta^j_{\ell'}, r_{\ell'})$ with $j \in \{1,2\}$, the expected number of vertices in $B_{p}(R) \cap B_O(\ell'+1) \setminus B_O(0.49 R)$ is again  $\Theta(\sum_{i=0.49R}^{\ell'+1}ne^{-(R-i)}e^{\frac12(R-\ell'-i)})=\Theta(1)$, with the constant hidden in the $\Theta(\cdot)$ notation proportional to $\nu$. The events of having no vertex in the mentioned neighborhoods of all anchor points are not independent, but they are positively correlated (conditional under having some regions empty, this only helps to have other regions empty). Hence, given that there are $\Theta(R)$ anchor points, the probability to have all desired regions empty (including the one of $v_1$) is at least $e^{-\Theta(R)}=n^{-\gamma}$, where $\gamma > 0$ can be made small by choosing $\lambda$ close to $1$. Define the \emph{right border} of the staircase of $v_1$ in the same way, and the probability that  all anchor points on both borders have their corresponding regions empty is at least $n^{-2\gamma}$. Moreover, the angle exposed by all these regions (outside $B_O(0.49R)$) is $\Theta\big( \sum_{\ell'=\ell-1}^{R-1} (e^{\frac12(R-\ell'-0.49R)}+e^{\frac12(R-2\ell')})\big)=\Theta(e^{\frac12(R-\ell-0.49R)})=\Theta(n^{0.51-\lambda})$, and for $C$ sufficiently large, all exposed area is inside a sector whose bisector is $\theta_{v_1}$ and whose angle is twice the angle of $\Upsilon_1$.

Next, partition $B_O(R)$ into $N=\Theta(n^{\lambda-0.51})$ sectors $\Xi_1,...,\Xi_N$ each of angle $3Cn^{0.51-\lambda}$ for some $C$ sufficiently large (the same $C$ as before, note that the angle of $\Xi_i$ is three times the angle of $\Upsilon_1$). Applying the above argument to the middle subsector $\Upsilon_i$ of the three subsectors of angle $Cn^{0.51-\lambda}$ of each sector $\Xi_i$, and noting that for $\lambda$ sufficiently close to $1$, we have $2\gamma < \lambda-0.51$, w.e.h.p. we find some $1 \le i \le N$ such that the corresponding middle subsector $\Upsilon_i$ is such that all desired regions corresponding to anchor points of the staircase around the starting vertex $v_i$ of $\Upsilon_i$ are empty. In that case, we claim that there is no edge crossing the (vertical or horizontal lines of the) staircase, and hence the component of $v_i$ is restricted to $\Xi_i$: indeed, suppose that there is a vertex $u$ in $B_O(\ell') \setminus B_O(\ell'-1)$ "below" the staircase ("below" refers to the following area: connect all anchor points starting from $(\theta^1_{\ell-1}, r_{\ell-1})$ both to the left and to the right by artificial lines in staircase manner, and the last one at radial distance $R-1$ via a straight line to the boundary; this divides $B_O(R)$ into two connected pieces, and "below" refers to the piece not containing the origin, and "above" to the piece containing the origin) that is connected by an edge to a vertex $w$ "above" the staircase.
Suppose first that $w$ is such that $\theta_w$ is within the smaller angle formed by the leftmost and rightmost anchor point of the staircase, and suppose w.l.o.g.~that $\theta_w$ is between the leftmost anchor point and $\theta_v$.
If $\theta_w$ is between $\theta^2_{\ell''}$ and $\theta^1_{\ell''}$, then $p=(\theta^2_{\ell''}, r_{\ell''})$
is, by Lemma~\ref{lem:FK15} applied with $p'=w$ and $p''=(\theta^1_{\ell''}, r_{\ell''})$  also at distance at most $R$ from $w$ (note that for any $\ell''$, the points $p$ and $p''$ are at distance less than $R$ by monotonicity of $\cosh$), which contradicts having the desired region empty. Otherwise, if $w$ is such that $\theta_w$ is not within the smaller angle formed by the leftmost and rightmost anchor point of the staircase, suppose w.l.o.g.~that $\theta_w$ is to the left of the leftmost anchor point of the staircase. Then, for
$r_w \in (r_{\ell''}, r_{\ell''+1})$ with $\ell'' < \ell'$,
the anchor point $(\theta^2_{\ell'-1}, r_{\ell'-1})$ is once again by Lemma~\ref{lem:FK15} also at distance at most $R$ from $w$, contradicting our assumption of having the desired region empty. If $\ell'' \ge \ell'$,
then we arrive at a contradiction:
on the one hand, by Lemma~\ref{lem:FK15} with $p'=(\theta^2_{\ell'-1}, r_{\ell'-1})$, $p''=u$ and $p=w$, the distance between $p'$ and $w$ is at most $R$. On the other hand,  $(\theta^2_{\ell'}, r_{\ell'})$ is at hyperbolic distance exactly $R$ from $(\theta^2_{\ell'-1}, r_{\ell'-1})$, and $w$ is in angular distance further away from $(\theta^2_{\ell'-1}, r_{\ell'-1})$ than $(\theta^2_{\ell'}, r_{\ell'})$ and it has also a strictly bigger radial coordinate than $r_{\ell'}$. Thus, by strict monotonicity of $\cosh$ (see Remark~\ref{rem:monotonTheta}) its hyperbolic distance is bigger than $R$, hence contradiction.
It follows that the component of $v_i$ is  inside a sector whose bisector is $\theta_{v_i}$ and whose angle is twice the angle of $\Upsilon_i$, and hence the component is inside $\Xi_i$.

Since in fact not only one, but w.e.h.p.~polynomially many such sectors $\Upsilon_i$ can be found, the argument shows that w.e.h.p.~polynomially many polynomial-size components exist (of size $\Omega(n^{\delta})$ for some $\delta > 0$), thus establishing Proposition~\ref{p:alphaOne}. Determining the exponent of the size of the second largest component remains open.

\section{Final remarks}
For $\frac12<\alpha<1$, the proof argument put forth 
  in this article does not seem 
  strong enough to be able to pinpoint the constant accompanying the
  $(\log n)^{\frac{1}{1-\alpha}}$ term in the asymptotic expression
  derived for $L_2(G)$ in Theorem~\ref{thm:main}.
We believe that developing techniques that would allow to do so
  is a worthwhile and interesting endeavor.



\bibliographystyle{alpha}
\bibliography{biblio} 

\end{document}